\documentclass[12pt]{amsart}
\usepackage{a4wide,graphicx,hyperref}

\newtheorem{proposition}{Proposition}[section]
\newtheorem{theorem}[proposition]{Theorem}
\newtheorem{lemma}[proposition]{Lemma}
\newtheorem{corollary}[proposition]{Corollary}
\newtheorem{question}[proposition]{Question}

\theoremstyle{remark} \theoremstyle{definition}
\newtheorem{example}[proposition]{Example}
\newtheorem{remark}[proposition]{Remark}
\newtheorem{definition}[proposition]{Definition}

\numberwithin{equation}{section}

\title{Similar dissection of sets}
\author[S.~Akiyama]{Shigeki Akiyama}
\author[J.~Luo]{Jun Luo}
\author[R.~Okazaki]{Ryotaro Okazaki}
\author[W.~Steiner]{Wolfgang Steiner}
\author[J.~Thuswaldner]{J\"org Thuswaldner}

\address{Department of Mathematics, Faculty of Science Niigata University, Ikarashi 2-8050, Niigata 950-2181, JAPAN}
\email{akiyama@math.sc.niigata-u.ac.jp \vspace{-1.5mm}}
\address{School of Mathematics and Computational Science, Sun Yat-Sen University, Guangzhou 510275, CHINA}
\email{luojun3@mail.sysu.edu.cn \vspace{-1.5mm}}
\address{Department of Knowledge Engineering and Computer Sciences, Doshisha University, 1-3 Tataramiyakodani, Kyotanabe-shi, Kyoto-fu, 610-0394 JAPAN}
\email{rokazaki@mail.doshisha.ac.jp \vspace{-1.5mm}}
\address{LIAFA, CNRS UMR 7089, Universit\'e Paris Diderot -- Paris 7,
Case 7014, 75205 Paris Cedex 13, FRANCE}
\email{steiner@liafa.jussieu.fr \vspace{-1.5mm}}
\address{Chair of Mathematics and Statistics, University of Leoben, Franz-Josef-Strasse 18, A-8700 Leoben, AUSTRIA}
\email{Joerg.Thuswaldner@unileoben.ac.at}
\thanks{This research was supported by the Japanese Ministry of Education, Culture, Sports, Science and Technology, Grant-in Aid for fundamental research 21540010, 2009--2011, by the project 10601069 of the National Natural Science Foundation of China, by the project ANR--06--JCJC--0073 ``DyCoNum'' of the French Agence Nationale de la Recherche, by the project S9610 of the Austrian Science Foundation, by the Amad\'ee grant FR--13--2008 and the PHC Amadeus grant 17111UB}

\begin{document}
\begin{abstract}
In 1994, Martin Gardner stated a set of questions concerning the dissection of a square or an equilateral triangle in three similar parts.
Meanwhile, Gardner's questions have been generalized and some of them are already solved.
In the present paper, we solve more of his questions and treat them in a much more general context.

Let $D\subset \mathbb{R}^d$ be a given set and let $f_1,\ldots,f_k$
be injective continuous mappings.
Does there exist a set $X$ such that $D = X \cup  f_1(X) \cup \ldots \cup f_k(X)$ is satisfied with a non-overlapping union?
We prove that such a set $X$ exists for certain choices of $D$ and $\{f_1,\ldots,f_k\}$.
The solutions $X$ often turn out to be attractors of iterated function systems with condensation in the sense of Barnsley.

Coming back to Gardner's setting, we use our theory to prove that an equilateral triangle can be dissected in three similar copies whose areas have ratio $1:1:a$ for $a \ge (3+\sqrt{5})/2$.
\end{abstract}

\maketitle

\begin{section}{Introduction}
In the present paper, we deal with the dissection of a given set $D$ into finitely many parts which are similar to each other.
Before we establish the fairly general setting of the present paper, we give a brief outline of the existing results on this topic.

In 1994, Martin Gardner~\cite{Gardner:94} (see also \cite[Chapter~16]{Gardner:01}) asked a set of questions concerning the dissection of a square as well as an equilateral triangle in three similar parts.
The existence of such a dissection is easy to verify if all parts are congruent to each other.
In the case of the square, we get three congruent rectangles.
Generalizing a result of Stewart and Wormstein~\cite{Stewart-Wormstein:92}, Maltby~\cite{Maltby:94} proved that this is the only dissection of a square in three congruent pieces (see also \cite{Liping-Yuqin-Ren:02}, where an analogous
question is settled for a parallelogram).

\begin{figure}[ht]
\includegraphics{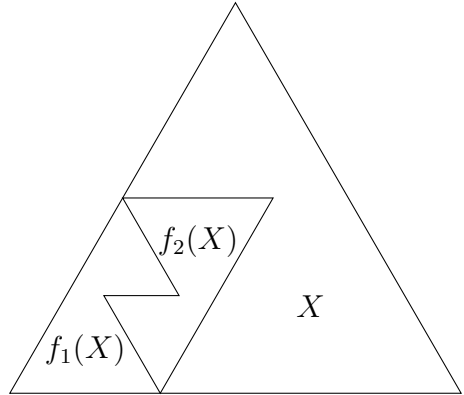}
\caption{Karl Scherer's dissection of an equilateral triangle in three similar parts, just two of which are congruent.}
\label{ShigekiPolygon}
\end{figure}

Finding a solution to Gardner's set of problems becomes more tricky if one requires that at least one of the parts is not congruent to the other ones.
A~nice solution to the problem of dissecting an equilateral triangle in three parts, just two of which are congruent, was given by Karl Scherer (see \cite[p.~123]{Gardner:01}).
It is depicted in Figure~\ref{ShigekiPolygon}.
Here, an equilateral triangle with vertices $(0,0), (1,0), (1/2,\sqrt{3}/2)$ is dissected into three pieces $X, f_1(X), f_2(X)$, where $f_1, f_2$ are two similarities with contractive ratios equal to $1/2$ and $X$ is the polygon with the consecutive vertices given by $(1/3,0), (1,0), (1/2,\sqrt{3}/2), (1/4,\sqrt{3}/4), (7/12,\sqrt{3}/4)$.
Scherer found nice solutions also for the case of dissecting a square with just two congruent parts.
Also, the dissection of a square as well as an equilateral triangle in three non-congruent pieces was done by him (all these dissections are depicted in \cite[Chapter~16]{Gardner:01}).
Chun, Liu and van Vliet~\cite{Chun-Liu-Vliet:96} studied a more general problem.
Indeed, let $m=a_1+\cdots+a_n$ be an integer composition of~$m$.
The question is to dissect a square in $m$ similar pieces so that there are $a_1$ pieces of largest size, $a_2$ pieces of second-largest
size and so on.
They prove that such a dissection is possible if and only if the composition is not of the form $m=(m-1)+1$.

In the present paper, we are going to generalize these questions considerably.
A first stage of generalization is contained in the following question, which will be solved partially in the subsequent sections and which will be used as a paradigm for our general theory.

\begin{question}
Can we dissect an equilateral triangle into three pieces of the same shape with area ratio $1:1:a$ for each $a>0$?
\end{question}

Contrary to the results quoted above, we want to gain solutions to dissection problems in similar parts whose similarity ratios are prescribed.
Indeed, using our general framework, we will be able to construct a dissection of an equilateral triangle in three pieces of area ratio $1:1:a$ with $a \in \{1\} \cup \big[\frac{3+\sqrt{5}}{2}, \infty\big)$.
Moreover, we will not restrict ourselves to the equilateral triangle but also consider arbitrary compact subsets of~$\mathbb{R}^d$.
Interestingly, in our studies we will meet the number ``high phi'' which is defined as the positive root of $x^3-2x^2+x-1$ (see \cite[p.~124]{Gardner:01}) and which already played a role in Scherer's original problems.
We mention that ``high phi'' is the square of the smallest Pisot number.

Note that the problem of finding a dissection with area ratios $1:1:a$ for arbitrary $a>0$ is trivial if we do not fix the set $D$ which we want to dissect.
For instance, as illustrated in Figure~\ref{rectangle}, for each
$r>0$, we can find a rectangle that admits an obvious dissection into three parts with area ratio $1:1:r^{-2}$.
Thus, throughout the present paper we are interested in finding dissections of a fixed set $D \subset \mathbb{R}^d$ in similar parts with prescribed ratios.

\begin{figure}[ht]
\includegraphics{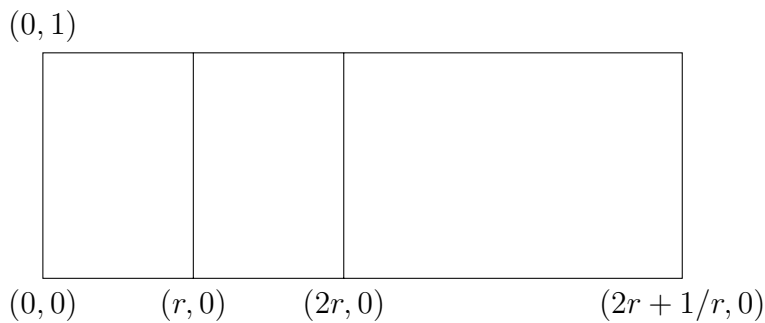}
\caption{For any given $r>0$, we can find a rectangle that can be dissected into three similar rectangles (each of whose side-lengths have ratio $r:1$) with area ratios $1:1:r^{-2}$.} \label{rectangle}
\end{figure}

We now set up a general framework that contains the problems discussed above as special cases.
In what follows, $\mu_d$ will denote the $d$-dimensional Lebesgue measure.

\begin{definition}
Let $D \subset \mathbb{R}^d$ be a compact set with $\overline{D^\circ} = D$ and $\mathcal{F} := \{f_1,f_2,\dots, f_k\}$ a finite family of injective mappings from $\mathbb{R}^d$ to itself.
We say that $\mathcal{F}$ admits a \emph{dissection of $D$} if there exists a compact set $X \subset \mathbb{R}^d$ with $\overline{X^\circ} = X$ such that
\[
D = X \cup f_1(X) \cup f_2(X) \cup \ldots \cup f_k(X),
\]
where $\mu_d(f_i(X) \cap f_j(X)) = 0$ for all disjoint $i,j \in \{1,\ldots,k\}$ and $\mu_d(X \cap f_i(X)) = 0$ for each $i \in \{1,\ldots,k\}$.
We call $X$ the \emph{generator} of the dissection.
\end{definition}

The difficulty of constructing a dissection of $D$ for a given family $\mathcal{F}$ depends on the properties of $\mathcal{F}$ and $D$.
Actually, one of our main aims is to discuss the existence and the uniqueness of the compact set $X$ so that $D = X \cup f_1(X) \cup \cdots \cup f_k(X)$ is a dissection of~$D$.
The treatment of the following classes turns out to be easier than the general case.

\begin{definition}
Let $D \subset \mathbb{R}^d$ be a compact set with
$\overline{D^\circ} = D$ and $\mathcal{F} := \{f_1,f_2,\dots, f_k\}$
a finite family of injective mappings from $\mathbb{R}^d$ to itself.
\begin{itemize}
\item
$\mathcal{F}$ is called \emph{inside family} (with respect to~$D$) if $f_i(D) \subset D$ for each $i \in \{1,\ldots,k\}$.
\item
$\mathcal{F}$ is called \emph{non-overlapping family} (with respect to~$D$) if $\mu_d\big(f_i(D) \cap f_j(D)\big) = 0$ holds for each $i,j \in \{1,\ldots,k\}$ with $i \neq j$.
\end{itemize}
\end{definition}

If all the functions in $\mathcal{F}$ are contractions, then $\mathcal{F}$ can be regarded as an \emph{iterated function system} (\emph{IFS} for short) in the sense of Hutchinson~\cite{Hutchinson:81}.
In this case, there exists a unique non-empty compact set $K \subset \mathbb{R}^d$, called the \emph{attractor} of the IFS $\mathcal{F}$, satisfying
\[
K = \bigcup_{i=1}^k f_i(K).
\]
Setting
\[
\Phi(X)= \bigcup_{i=1}^k f_i(X),
\]
this can be written as $K = \Phi(K)$.

A variant of IFS are so-called \emph{IFS with condensation} ({\it cf.} Barnsley~\cite{Barnsley:93}).

\begin{definition}[IFS with condensation]
Let $\mathcal{F} = \{f_1,\ldots, f_k\}$ be a family of contractions in $\mathbb{R}^d$ and $A \subset \mathbb{R}^d$ a nonempty compact set.
Then the pair $(\mathcal{F},A)$ is called an \emph{IFS with condensation~$A$}.
The unique non-empty compact set $K$ satisfying the set equation
\[
K= A \cup f_1(K) \cup \ldots \cup f_k(K) = A \cup \Phi(K)
\]
is called the \emph{attractor} of the IFS $\mathcal{F}$ with  condensation~$A$.
\end{definition}

The unique existence of $K$ is proved by a standard fix point argument.
It is given by
\[
K = A \cup \Phi(A) \cup \Phi^2(A) \cup \ldots
\]
In some cases, the solution $X$ to our dissection problem will be an attractor of an IFS related to $\mathcal{F}$ with a certain condensation depending on the set~$D$.
\end{section}

\begin{section}{Non-overlapping inside families}
We start with the easiest case, non-overlapping inside families.
We can construct a dissection for these families provided that $\mathcal{F}$ consists of contractions and $D$ is compact.
The main result of this section, Theorem~\ref{IFS}, will be used in subsequent sections in order to settle more complicated cases.

For the proof of Theorem~\ref{IFS}, we need the following consequence of the invariance of domains.

\begin{lemma}\label{Brouwer}
Let $f:\, \mathbb{R}^d \to \mathbb{R}^d$ be an injective contraction
and $X \subset \mathbb{R}^d$ a compact set. Then $\partial f(X) =
f(\partial X)$.
\end{lemma}

\begin{proof}
By the invariance of domains (see {\it e.g.} \cite[Theorem~2B.3]{Hatcher:02}), the mapping $f$ is a homeomorphism.
This implies the result.
\end{proof}

\begin{theorem}\label{IFS}
Let $D \subset \mathbb{R}^d$ be a compact set with $D = \overline{D^\circ}$ and $\mu_d(\partial D) = 0$.
Let $\mathcal{F} := \{f_1,\ldots,f_k\}$ be an IFS on~$\mathbb{R}^d$, whose attractor is denoted by~$E$.
Suppose that $\mathcal{F}$ is a non-overlapping inside family.
Then $\mathcal{F}$ admits a dissection of $D$ if and only if $\mu_d(E) = 0$.
Moreover, the generator of the dissection is unique.
\end{theorem}

\begin{proof}
Assume that $\mathcal{F}$ admits a dissection of~$D$ generated by
$X$. Then we have $X^\circ \cap \Phi(X^\circ) = \emptyset$, thus the
non-overlapping inside property implies that $\Phi(X^\circ) \cap
\Phi^2(X^\circ) = \emptyset$, which yields $\Phi^2(X^\circ) \subset
D^\circ \setminus \Phi(X^\circ)$. Since $\overline{X^\circ} = X$ and
$X \cup \Phi(X) = D$, we obtain that $\Phi^2(X) \subset X$. By
induction, we see that
\begin{equation} \label{Phi2nX}
\Phi^{2n}(X) \subset X\quad \mbox{for all}\ n\ge0.
\end{equation}

Set $Y := \overline{D \setminus \Phi(D)}$.
Then we have $Y \subset X$ and, by (\ref{Phi2nX}),
\begin{equation} \label{YX}
Y \cup \Phi^2(Y) \cup \Phi^4(Y) \cup \ldots \subset X.
\end{equation}
Hutchinson's classical theory on IFS (see~\cite{Hutchinson:81}) implies that $(\Phi^{2n}(Y))_{n\ge0}$ converges to $E$ in Hausdorff metric.
Since $X$ is closed, we obtain
\begin{equation} \label{EX}
E \subset X.
\end{equation}
By definition, we have $E = \Phi(E)$, implying that $E \subset X \cap \Phi(X)$ and, hence, $\mu_d(E) \le \sum_{i=1}^k \mu_d\big(X\cap f_i(X)\big) = 0$.
Therefore, $\mathcal{F}$ can admit a dissection of $D$ only if $\mu_d(E) = 0$.

\begin{figure}[ht]
\includegraphics{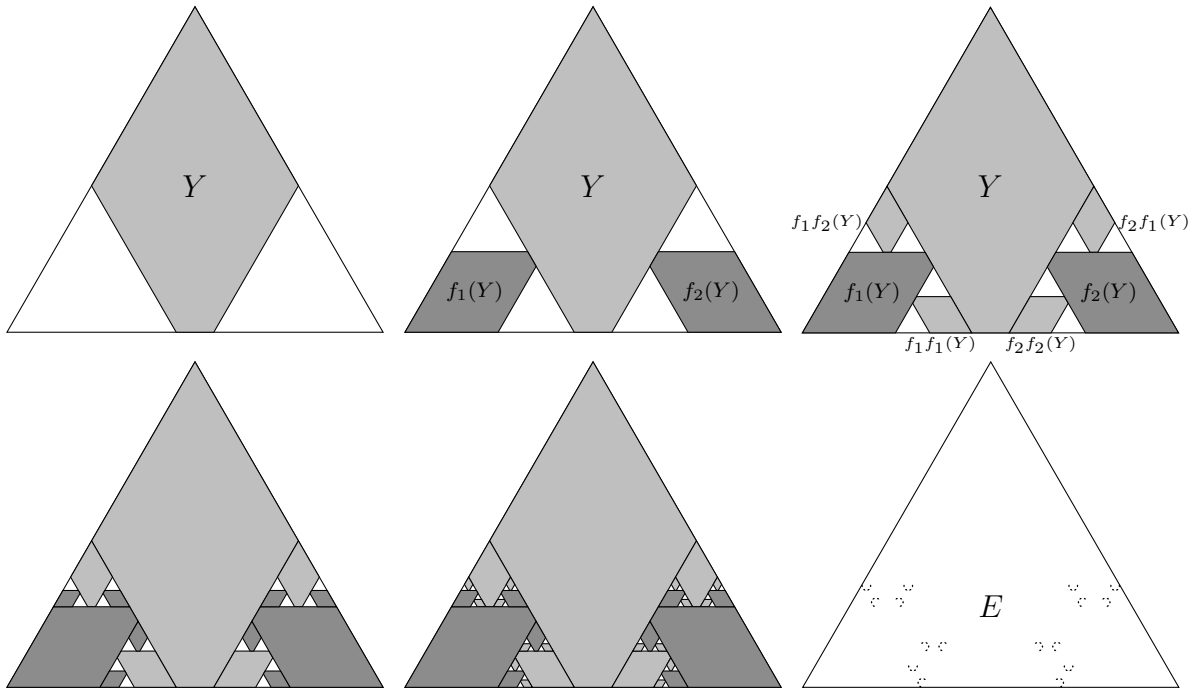}
\caption{The attractor $E$ of the IFS with an inside non-overlapping family $\{f_1,f_2\}$ (below right) and the convergence process in the IFS with condensation $\big(\{f_1f_1,f_1f_2,f_2f_1,f_2f_2\}, Y \cup E\big)$, $Y := \overline{D \setminus \Phi(D)}$, to the generator of the dissection of $D$ with respect to $\{f_1,f_2\}$.
The mappings $f_1,f_2$ are given in Example~\ref{exa4}, with $r=9/20$.}
\label{approximation}
\end{figure}

Assume now that $\mu_d(E)=0$ and consider the set
\[
Z := E \cup Y \cup \Phi^2(Y) \cup \Phi^4(Y) \cup \ldots
\]
with $Y$ defined as above. Note that
\begin{equation} \label{Zint}
Z = \overline{Z^\circ}
\end{equation}
because $\Phi^{2n}(Y) = \overline{\Phi^{2n}(Y)^\circ}$ (by Lemma~\ref{Brouwer}) and $E$ is the Hausdorff limit of the sequence of sets $(\Phi^{2n}(Y))_{n\ge0}$.
By the non-overlapping condition, we have
\begin{equation} \label{fijZ}
\mu_d\big(f_i(Z) \cap f_j(Z)\big) = 0\quad \mbox{for}\ i \ne j.
\end{equation}
Next, we shall prove that
\begin{equation} \label{muY}
\mu_d\big(\Phi^n(Y) \cap \Phi^m(Y)\big) = 0\quad \mbox{for}\ n \neq m.
\end{equation}
See Figure~\ref{approximation} for an illustration of $Y, \Phi(Y), \Phi^2(Y), \Phi^3(Y), \Phi^4(Y)$ and~$E$.
As $\mu_d(\partial D) = 0$,
\[
\mu_d\big(\Phi^n(Y) \cap \Phi^m(Y)\big) = \mu_d\big(\Phi^n(D \setminus \Phi(D)) \cap \Phi^m(D \setminus \Phi(D))\big).
\]
By the non-overlapping condition, we obtain
\begin{equation} \label{nmY}
\mu_d\big(\Phi^n(Y) \cap \Phi^m(Y)\big) = \mu_d\Big(\big(\Phi^n(D) \setminus \Phi^{n+1}(D)\big) \cap \big(\Phi^m(D) \setminus \Phi^{m+1}(D)\big)\Big).
\end{equation}
Since $\Phi(D^\circ) \subset D^\circ$ by the inside condition, we know that $\Phi^{n+1}(D^\circ) \subset \Phi^{n}(D^\circ)$, thus the sets $\Phi^{n}(D^\circ) \setminus \Phi^{n+1}(D^\circ)$ are pairwise disjoint, hence the right hand side of (\ref{nmY}) is 0, which yields (\ref{muY}).
Thus, as $\mu_d(E) = 0$ by assumption, we get
\begin{equation} \label{ZPhiZ}
\mu_d\big(Z \cap \Phi(Z)\big) \le \mu_d(E) + \mu_d\big(\Phi(E)\big) = 0.
\end{equation}
Moreover, since $Y \cup \Phi(Y) \cup \ldots \cup \Phi^n(Y) = \overline{D \setminus \Phi^{n+1}(D)}$, which tends to $\overline{D \setminus E}$ in Hausdorff metric, we get
\begin{equation} \label{equalD}
Z \cup \Phi(Z) = E \cup Y \cup \Phi(Y) \cup \Phi^2(Y) \cup \ldots = D.
\end{equation}
Combining (\ref{Zint}), (\ref{fijZ}), (\ref{ZPhiZ}) and (\ref{equalD}), we conclude that $\mathcal{F}$ admits a dissection of~$D$.
Two examples for dissections originating from non-overlapping inside families are given in Figure~\ref{approximation2} (see also Example~\ref{exa4}).

\begin{figure}[ht]
\includegraphics{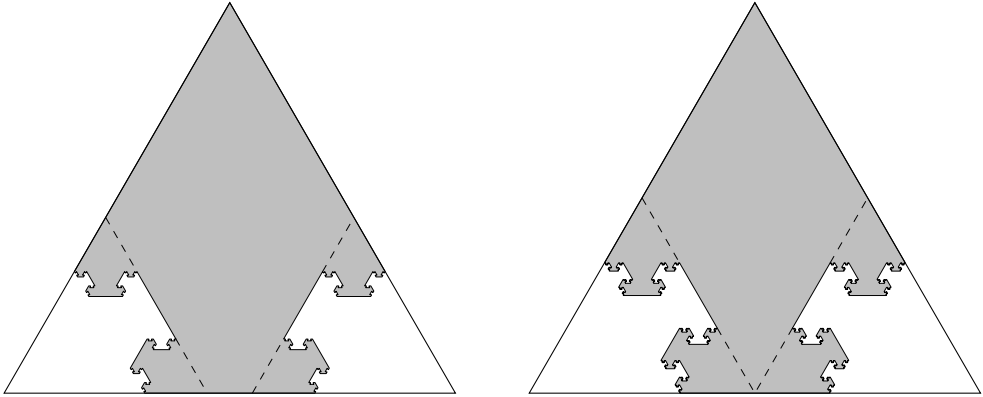}
\caption{The dissections discussed in Example~\ref{exa4} for $r=9/20$ and $r=1/2$.}
\label{approximation2}
\end{figure}

To prove the uniqueness of the generator of a dissection, assume
that $X$ generates a dissection of $D$, which is different from~$Z$.
Then (\ref{YX}) and (\ref{EX}) imply that $Z \subset X$. As $Z =
\overline{Z^\circ}$ and $X = \overline{X^\circ}$, we obtain $\mu_d(X
\setminus Z) > 0$. By the dissection property of~$Z$, we have $X
\setminus Z \subset \Phi(Z) \subset \Phi(X)$, thus $\mu\big(X \cap
\Phi(X)\big) \ge \mu(X \setminus Z) > 0$, which contradicts the
dissection property of~$X$.
\end{proof}

If all the $f_i$ are similarities, the condition $\mu_d(E) = 0$ can be checked easily.

\begin{corollary} \label{simcor}
Let $D \subset \mathbb{R}^d$ be a compact set with $D = \overline{D^\circ}$ and $\mu_d(\partial D) = 0$.
Let $\mathcal{F} :=\{f_1,\ldots,f_k\}$ be an IFS on~$\mathbb{R}^d$, where every $f_i$ is a similarity.
Suppose that $\mathcal{F}$ is a non-overlapping inside family and $\Phi(D) \neq D$.
Then $\mathcal{F}$ admits a unique dissection of~$D$.
\end{corollary}

\begin{proof}
By Theorem~\ref{IFS}, we only have to show that $\mu_d(E) = 0$.
Observe that $\Phi(D) \neq D$ implies that
\[
\mu_d(D) > \mu_d\big(\Phi(D)\big) = \sum_{i=1}^k r_i^d \mu_d(D),
\]
where $r_i$ is the contraction ratio of $f_i$ for each $i \in
\{1,\ldots,k\}$, thus $\sum_{i=1}^k r_i^d < 1$. Since $E = \Phi(E)$,
we have
\[
\mu_d(E) = \mu_d\big(\Phi(E)\big) = \sum_{i=1}^k r_i^d \mu_d(E),
\]
which implies $\mu_d(E) = 0$.
\end{proof}

From the definition of $Z$ in the proof of Theorem~\ref{IFS}, we
obtain the following description of the dissection in terms of an
IFS with condensation.

\begin{corollary}
Let $D$ and $\mathcal{F}$ be given as in Theorem~\ref{IFS}.
Let $E$ be the attractor of the IFS~$\mathcal{F}$, $\mu_d(E) = 0$.
Then the unique dissection of $D$ with respect to $\mathcal{F}$ is given by the unique solution of the IFS with condensation $\big(\{f \circ g:\, f,g \in \mathcal{F}\}, \overline{D \setminus \Phi(D)} \cup E\big)$.
\end{corollary}

In the following, we discuss some examples for Theorem~\ref{IFS}.
\begin{corollary}
Subdividing a star body by the ratio $1:a$ is possible for any $a>0$.
\end{corollary}

\begin{proof}
Let $D \subset \mathbb{R}^2$ be a star body and suppose the origin
is the center of this star body, {\it i.e.}, any segment connecting
the origin and a point in $D$ is entirely in $D$. Take $\mathcal{F}
= \{f_1\}$ with $f_1(x) = x/\sqrt{a}$. Then $D$ and $\mathcal{F}$
meet the conditions of Theorem~\ref{IFS}. This proves the corollary.
\end{proof}

Figure~\ref{f1} shows an example for a subdivision of the equilateral triangle with area ratio $1:2$.

\begin{figure}[ht]
\includegraphics{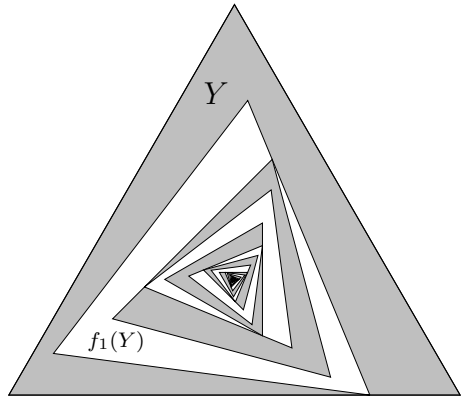}
\caption{A dissection of the triangle with area ratio $1:2$.
The dissection is done by the IFS $\{f_1\}$ with $f_1(x,y) = \sqrt{1/2} R(5\pi/8) (x,y) + (4/5,0)$ where $R(\alpha)$ denotes the counterclockwise rotation with angle~$\alpha$; $Y = \overline{D \setminus f_1(D)}$.}
\label{f1}
\end{figure}

\begin{example}\label{exa4}
We want to dissect the equilateral triangle $D = \triangle\big((0,0), (1,0), \big(\frac{1}{2}, \frac{\sqrt{3}}{2}\big)\big)$ with the IFS $\{f_1,f_2\}$, where
\[
f_1(x,y) = r R\Big(\frac{2\pi}{3}\Big) (x,y) + (r,0)\quad \hbox{and}\quad f_2(x,y) = r R\Big(\frac{4\pi}{3}\Big) (x,y) + \Big(1-\frac{r}{2}, \frac{r\sqrt{3}}{2}\Big),
\]
with $r \in (0,1/2]$ and $R(\alpha)$ being the counterclockwise rotation with angle $\alpha$ around the origin ({\it cf.} Figure~\ref{approximation}).
It is easy to see that $\{f_1,f_2\}$ is an inside non-overlapping family, providing a dissection of $D$ with respect to $\{f_1,f_2\}$ in view of Corollary~\ref{simcor}.
Figure~\ref{approximation2} shows the dissections for the choices $r=9/20$ and $r = 1/2$.
\end{example}

The mappings defined in Example~\ref{exa4} show that the equilateral triangle can be dissected in similar parts with area ratios $1:1:a$ for each $a \ge 4$.
Figure~\ref{approximation2} suggests that it is possible to go beyond this bound.
Indeed, we will establish dissections coming from families where the inside as well as the non-overlapping condition will be violated.
An application will be the construction of dissections of the equilateral triangle with area ratio $1:1:a$ for each $a \ge (3+\sqrt{5})/2$.
\end{section}

\begin{section}{A general dissection result and its consequences}
In this section we will give a criterion which enables us to construct dissections of $D$ with respect to (not necessarily inside and non-overlapping) families~$\mathcal{F}$.

\begin{theorem} \label{main}
Let $D \subset \mathbb{R}^d$ be a compact set with $D = \overline{D^\circ}$ and $\mu_d(\partial D) = 0$.
Let $\mathcal{F} :=\{f_1,\ldots,f_k\}$ be an IFS on~$\mathbb{R}^d$, whose attractor $E$ satisfies $\mu_d(E) = 0$.

Suppose that there exists some $Y \subset D$ satisfying the following conditions.
\begin{itemize}
\item[(1)]
$\overline{Y^\circ} = Y$, $\mu_d(\partial Y) = 0$,
\item[(2)]
$Y, f_1(Y), \ldots, f_k(Y)$ are subsets of $D$ which are mutually disjoint in measure,
\item[(3)]
$Y, f_1\big(D \setminus \Phi(Y)\big), \ldots, f_k\big(D \setminus \Phi(Y)\big)$ are subsets of $D$ which are mutually disjoint in measure.
\end{itemize}
Then $D$ admits a dissection with respect to~$\mathcal{F}$.
\end{theorem}

\begin{proof}
Let $C = \overline{D \setminus \big(Y \cup \Phi(Y)\big)}$.
Obviously, the set $C$ is compact and the closure of its interior.
Moreover, we have $\mu_d(\partial C)=0$. We only have to show that
$\mathcal{F}$ is an inside non-overlapping family for~$C$. By
Theorem~\ref{IFS}, this implies that $\mathcal{F}$ admits a unique
dissection of $C$ generated by $X$. From this immediately follows
together with (2) that $X \cup Y$ generates a dissection of $D$ with
respect to~$\mathcal{F}$.

We first prove that $\mathcal{F}$ is an inside family for $C$, {\it
i.e.}, that $f_i(C) \subset C$. By (3), we have
\begin{equation}\label{new1}
\mu_d\big(Y \cap f_i\big(D\setminus \Phi(Y)\big)\big) = 0.
\end{equation}
Moreover, since by (2) we have $Y \subset D\setminus\Phi(Y)$ up to a set of measure zero, (3) implies that
\begin{equation}\label{new2}
\mu_d\big(f_j(Y) \cap f_i\big(D \setminus \Phi(Y)\big)\big) = 0
\quad (i\not=j).
\end{equation}
By the injectivity of $f_i$ we have
\begin{equation}\label{new3}
\mu_d\big(f_i(Y) \cap f_i(D \setminus Y)\big) = 0.
\end{equation}
Combining \eqref{new1}, \eqref{new2} and \eqref{new3} we arrive at
\begin{equation}\label{new4}
\mu_d\Big(\big(Y \cup \Phi(Y)\big) \cap f_i\big(D \setminus \big(Y \cup \Phi(Y)\big)\big)\Big) = 0.
\end{equation}
Together with (1) and the fact that $\mu_d(\partial D)=0$ equation
\eqref{new4} yields
\[
\mu_d\Big(\big(Y \cup \Phi(Y)\big) \cap f_i(C)\Big) = 0.
\]
Applying (3) again, we conclude that $f_i(C) \subset C$, thus
$\mathcal{F}$ is an inside family for~$C$.

The non-overlapping property of $\mathcal{F}$ for $C$ follows from
(3) as well. This proves the theorem.
\end{proof}

\begin{remark} \label{3dash}
In view of Condition~(2), Condition~(3) can be rewritten as
\begin{itemize}
\item[(3')]
$f_1(C), \ldots, f_k(C)$ are subsets of $C$ which are mutually disjoint in measure.
\end{itemize}
Here, $C = \overline{D \setminus \big(Y \cup \Phi(Y)\big)}$ as in the proof of Theorem~\ref{main}.
\end{remark}

\begin{remark}
If $\mathcal{F}$ is a non-overlapping inside family, then we can choose $Y = \emptyset$ in Theorem~\ref{main}.
\end{remark}

\begin{corollary}\label{outsidecorollary}
Let $D \subset \mathbb{R}^d$ be a compact set with $D = \overline{D^\circ}$ and $\mu_d(\partial D) = 0$.
Let $\mathcal{F} :=\{f_1,\ldots,f_k\}$ be an IFS on~$\mathbb{R}^d$, whose attractor $E$ satisfies $\mu_d(E) = 0$.

If $\mathcal{F}$ is a non-overlapping family,
\begin{itemize}
\item[(i)]
$\Phi\big(D \setminus \Phi(D)\big) \subset D$ and
\item[(ii)]
$\Phi\big(D \cap \Phi^2(D)\big) \subset D$,
\end{itemize}
then $D$ admits a dissection with respect to~$\mathcal{F}$.
\end{corollary}

\begin{proof}
We show that $Y := \overline{D \setminus \Phi(D)}$ satisfies
Conditions (1)--(3) of Theorem~\ref{main}. Since $D =
\overline{D^\circ}$ and $\mu_d(\partial D) = 0$, the same properties
hold for~$Y$, which implies Condition~(1). Condition~(2) is an
immediate consequence of (i), the non-overlapping property and the
fact that
\[
\mu_d\big(Y \cap \Phi(Y)\big) \le \mu_d\big(Y \cap \Phi(D)\big)=0.
\]
To show Condition~(3), note that
\begin{equation} \label{DY}
D \setminus \Phi(Y) = D \setminus \Phi\big(\overline{D \setminus \Phi(D)}\big) \subset D \setminus \Phi\big(D \setminus \Phi(D)\big) \subset D \setminus \big(\Phi(D) \setminus \Phi^2(D)\big).
\end{equation}
Since
\begin{equation} \label{DDD2}
D \setminus \big(\Phi(D) \setminus \Phi^2(D)\big) = \big(D \setminus \Phi(D)\big) \cup \big(D \cap \Phi^2(D)\big),
\end{equation}
using (\ref{DY}), (\ref{DDD2}), (i) and (ii), we obtain that
\[
\Phi\big(D \setminus \Phi(Y)\big) \subset \Phi\big(D \setminus \Phi(D)\big) \cup \Phi\big(D \cap \Phi^2(D)\big) \subset D.
\]
Together with the non-overlapping property and the fact that
\[
\mu_d\big(Y \cap \Phi\big(D\setminus \Phi(Y)\big)\big) \le
\mu_d\big(Y \cap \Phi(D)\big) = 0,
\]
this implies Condition~(3).
\end{proof}

\begin{remark}
For each positive integer $n$, we can replace Conditions (i) and (ii) in Corollary~\ref{outsidecorollary} by
\begin{itemize}
\item[(i)]
$\Phi\big(D \cap \Phi^{2k}\big(D \setminus \Phi(D)\big)\big) \subset D$ for all $0 \le k < n$ and
\item[(ii)]
$\Phi\big(D \cap \Phi^{2n}(D)\big) \subset D$.
\end{itemize}
To prove this, set
\[
Y := \overline{D\cap \bigcup_{0\le k<n} \Phi^{2k}\big(D \setminus
\Phi(D)\big)}.
\]
Again, we have to show that Conditions (1)--(3) of Theorem~\ref{main} are fulfilled.
Conditions~(1) and (2) are proved as for Corollary~\ref{outsidecorollary}.
Condition~(3) follows now from
\begin{align*}
D \setminus \Phi(Y) & \subset D \setminus \bigcup_{0\le k<n} \Phi^{2k+1}\big(D \setminus \Phi(D)\big) \\
& \subset D \setminus \bigcup_{0\le k<n} \Big(\Phi^{2k+1}(D) \setminus \Phi^{2k+2}(D)\Big) \\
& \subset D \cap \bigg(\bigcup_{0\le k<n} \Big(\Phi^{2k}(D) \setminus \Phi^{2k+1}(D)\Big) \cup \Phi^{2n}(D)\bigg) \\
& = \bigcup_{0\le k<n} \Big(D \cap \Phi^{2k}\big(D \setminus \Phi(D)\big)\Big) \cup \Big(D \cap \Phi^{2n}(D)\Big)
\end{align*}
by the same reasoning as in the proof of Corollary~\ref{outsidecorollary}.
\end{remark}
\end{section}

\begin{section}{Examples for general dissections}

The following example shows that not every overlapping family yields a dissection.
We do not have a satisfactory answer for the existence of a solution, see Section~\ref{problems}.

\begin{example} \label{exanodis}
Let $D = \triangle\big((0,0), (1,0), \big(\frac{1}{2},
\frac{\sqrt{3}}{2}\big)\big)$ and the IFS $\{f_1,f_2\}$ given by
\[
f_1(x,y) = r R\Big(\frac{4\pi}{3}\Big) (x,y) + \Big(\frac{r}{2}, \frac{r\sqrt{3}}{2}\Big),\qquad
f_1(x,y) = r R\Big(\frac{2\pi}{3}\Big) (x,y) + (1,0)
\]
with $r>1/2$.
Then we have $\mu_2\big(f_1(Y) \cap f_2(Y)\big) > 0$ for $Y := \overline{D \setminus \Phi(D)}$, see Figure~\ref{fignodis}.
Since every dissection of $D$ with respect to $\{f_1,f_2\}$ must contain~$Y$, we conclude that $D$ admits no such dissection.
\end{example}

\begin{figure}[ht]
\includegraphics{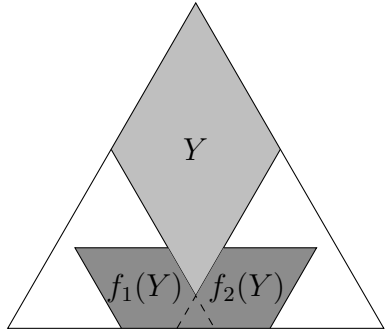}
\caption{The sets $Y, f_1(Y), f_2(Y)$ showing that Example~\ref{exanodis} admits no dissection.}
\label{fignodis}
\end{figure}

As a first application of Theorem~\ref{main}, we extend Example~\ref{exa4} to contraction ratios $r \le (\sqrt{5}-1)/2$.

\begin{example} \label{exagold}
Let $D = \triangle\big((0,0), (1,0), \big(\frac{1}{2}, \frac{\sqrt{3}}{2}\big)\big)$ and the IFS $\{f_1,f_2\}$ given by
\[
f_1(x,y) = r R\Big(\frac{2\pi}{3}\Big) (x,y) + (r,0),\qquad f_2(x,y) = r R\Big(\frac{4\pi}{3}\Big) (x,y) + \Big(1-\frac{r}{2}, \frac{r\sqrt{3}}{2}\Big)
\]
with $r \in \big(0, \frac{\sqrt{5}-1}{2}\big]$.
Choose
\[ Y := \triangle\Big(\Big(\frac{1}{2(1+r)}, \frac{\sqrt{3}}{2(1+r)}\Big), \Big(1-\frac{1}{2(1+r)}, \frac{\sqrt{3}}{2(1+r)}\Big), \Big(\frac{1}{2}, \frac{\sqrt{3}}{2}\Big)\Big). \]
Then it is easy to see that $Y$ satisfies the conditions of Theorem~\ref{main}, see Figure~\ref{Cfigure}.
Indeed, the sets $f_1(C)$ and $f_2(C)$ are disjoint in measure if and only if the first coordinate of the rightmost point of $f_1(C)$ is less than or equal to $1/2$.
As this yields the inequality $\frac{r}{1+r} + \frac{1}{2} \frac{r^2}{1+r} \le \frac{1}{2}$, we obtain the condition $r \le (\sqrt{5}-1)/2$.
Figure~\ref{goldenmean} shows the dissections for the choices $r=11/20$ and $r=(\sqrt{5}-1)/2$.
\end{example}

\begin{figure}[ht]
\includegraphics{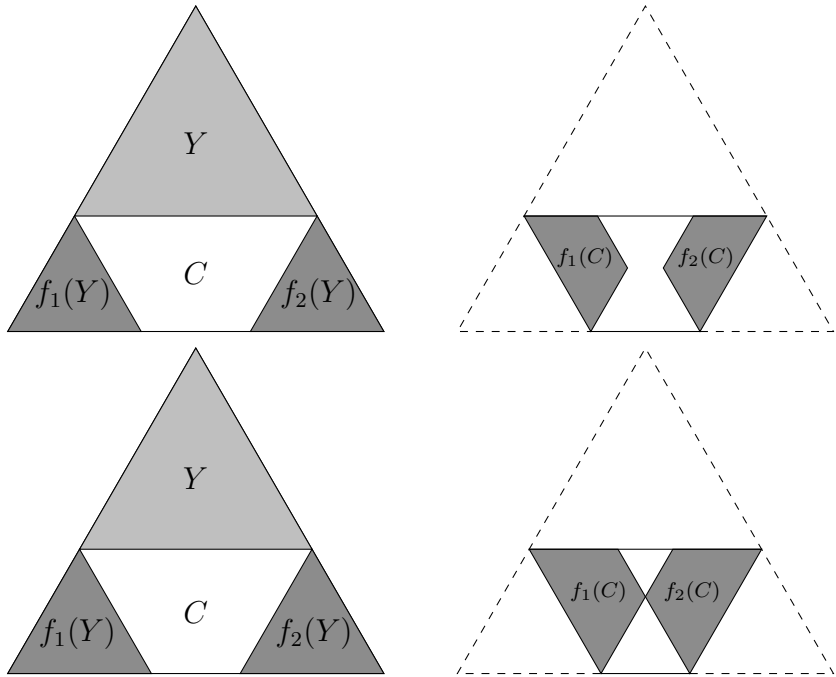}
\caption{The sets $Y, f_1(Y), f_2(Y), C, f_1(C), f_2(C)$ for the choices $r = 11/20$ (above) and $r = (\sqrt{5}-1)/2$ (below) in Example~\ref{exagold}.
The pictures on the left show Condition~(2) of Theorem~\ref{main}.
The pictures on the right show Condition~(3') of Remark~\ref{3dash}.}
\label{Cfigure}
\end{figure}

\begin{figure}[ht]
\includegraphics{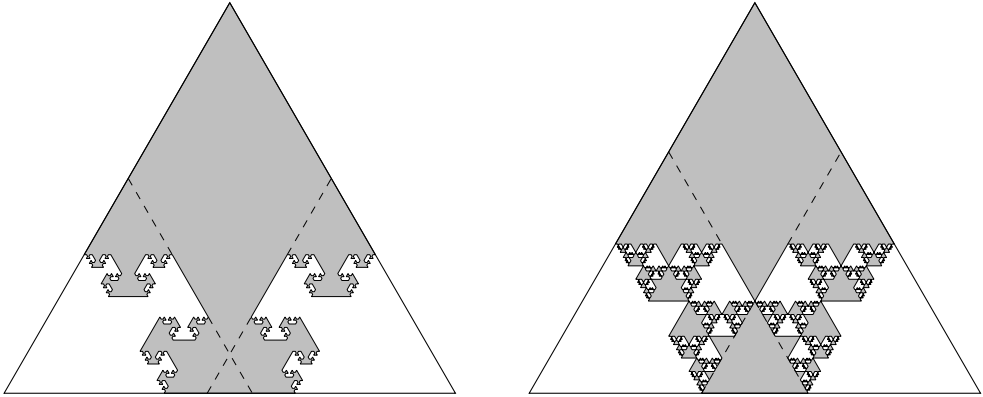}
\caption{Dissections from Example~\ref{exagold} for $r=11/20$ and $r=(\sqrt{5}-1)/2$.}
\label{goldenmean}
\end{figure}

Note that the boundary of this dissection can be described by another IFS with condensation: $B = f_1(B) \cup \big[\frac{r}{1+r},\frac{1}{1+r}\big] \cup f_2(B)$.
The solution $B$ satisfies $f_j(B)= C \cap f_j(C)$ for $j=1,2$ and one can show that $B$ is a simple arc for $r > (\sqrt{5}-1)/2$.
The limit case $r = (\sqrt{5}-1)/2$ is of interest because $f_1(B) = C \cap f_1(C)$ coincides with an IFS attractor associated to the set of uniqueness $\mathcal{U}_{(\sqrt{5}-1)/2}$ in the golden gasket (see p.~1470--1472 in~\cite{Sidorov}).

Example~\ref{exagold} shows that the equilateral triangle can be dissected into similar parts with area ratio $1:1:a$ with $a \ge (3+\sqrt{5})/2$.

\begin{example} \label{exaflip}
Let $D = \triangle\big((0,0), (1,0), \big(\frac{1}{2},
\frac{\sqrt{3}}{2}\big)\big)$ and the IFS $\{f_1,f_2\}$ given by
\[
f_1(x,y) = r R\Big(\frac{\pi}{3}\Big) (x,-y),\qquad f_2(x,y) = r R\Big(\frac{4\pi}{3}\Big) (x,y) + \Big(1-\frac{r}{2}, \frac{r\sqrt{3}}{2}\Big)
\]
with $r \in (0, 1/\phi]$, where $\phi$ denotes ``high phi'', the
positive root of $x^3-2x^2+x-1$. Choose
\[ Y := \triangle\Big(\Big(\frac{r}{2}, \frac{r\sqrt{3}}{2}\Big), \Big(1-\frac{r}{2}, \frac{r\sqrt{3}}{2}\Big), \Big(\frac{1}{2}, \frac{\sqrt{3}}{2}\Big)\Big). \]
Then $Y$ satisfies the conditions of Theorem~\ref{main}, see
Figure~\ref{Cfigure2}. We have $f_2(C) \subset C$ if and only if the
$x$-coordinate of the leftmost point of $f_2(C)$ is at least $1/2$,
{\it i.e.}, $r - \frac{r^2(1-r)}{2} \ge \frac{1}{2}$.
Figure~\ref{flip} shows the dissections for $r=1/2$ and $r=1/\phi$.
Note that $X$ is not connected for $r < 1/\phi$.
\end{example}

\begin{figure}[ht]
\includegraphics{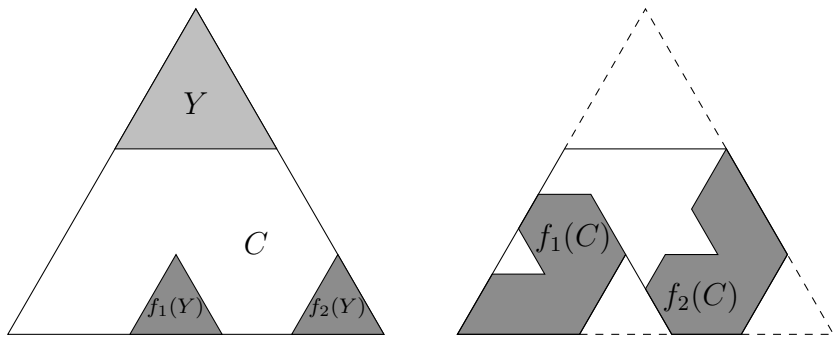}
\caption{The sets $Y, f_1(Y), f_2(Y), C, f_1(C), f_2(C)$ for $r=1/\phi$ in Example~\ref{exaflip}.}
\label{Cfigure2}
\end{figure}

\begin{figure}[ht]
\includegraphics{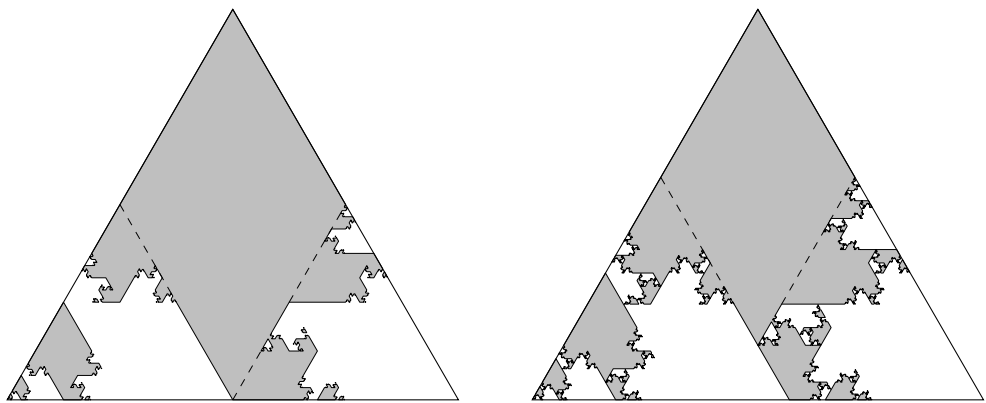}
\caption{Dissections from Example~\ref{exaflip} for $r=1/2$ and $r=1/\phi$.}
\label{flip}
\end{figure}

\begin{example} \label{exasquare}
Let now $D$ be the unit square. By Theorem \ref{IFS}, one can
construct many different dissections of ratio $1:1:a$ with $a \ge
4$. However, for $\phi^2 \le a < 4$ ($\phi$ denotes ``high phi''
again), the following construction is basically the only one that we
know, and we do not know any dissection with $1<a< \phi^2$. Let
\[
f_1(x,y) = r (x,y),\qquad f_2(x,y) = r R\Big(\frac{3\pi}{2}\Big) (x,y) + (1-r,1).
\]
Figure~\ref{figsquare} shows that $Y :=
\overline{D\setminus\Phi(D)}$ satisfies the conditions of
Theorem~\ref{main} for $r \le 1/\phi$. Indeed, we need $\mu_2(f_1(C)
\cap f_2(Y)) = 0$ in order to have $f_1(C) \subset C$, and this
holds if and only if the rightmost point in $C$ of the form $(x,1)$
satisfies $r x \le 1-r$, {\it i.e.}, $r(1-r(1-r)) \le 1-r$. The
dissections for $r=1/2$ and the limit case $r=1/\phi$ are depicted
in Figure~\ref{figsquaredis}.
\end{example}

\begin{figure}
\includegraphics{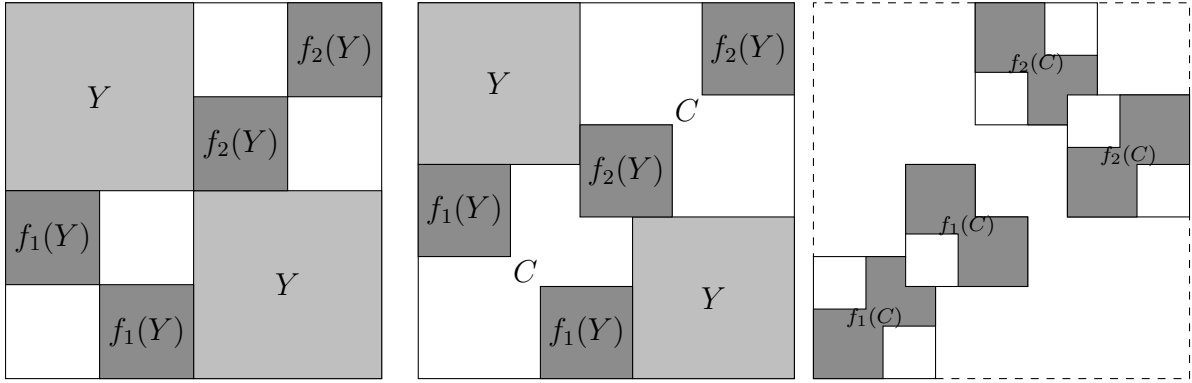}
\caption{Two instances of Example~\ref{exasquare}:
$Y=\overline{D\setminus\Phi(D)}, f_1(Y), f_2(Y)$ for the
non-overlapping case $r=1/2$ (left),
$Y=\overline{D\setminus\Phi(D)}, f_1(Y), f_2(Y),
C=\overline{D\setminus(Y\cup\Phi(Y))}, f_1(C), f_2(C)$ for the
overlapping case $r=1/\phi$ (middle and right).} \label{figsquare}
\end{figure}

\begin{figure}
\includegraphics{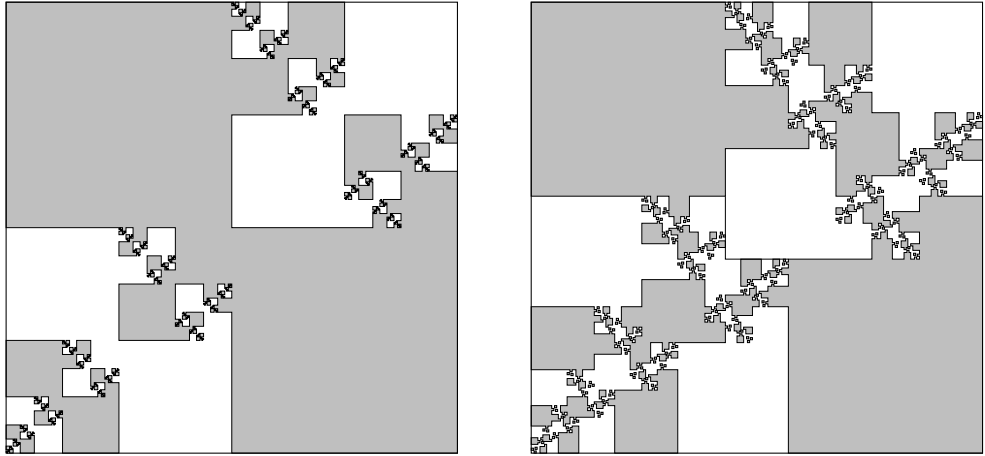}
\caption{Dissections from Example~\ref{exasquare} for $r=1/2$ (left) and $r=1/\phi$ (right).}
\label{figsquaredis}
\end{figure}

The last example concerns non-overlapping outside families of the equilateral triangle.

\begin{example} \label{exaoutside}
Let $D = \triangle\big((0,0), (1,0), \big(\frac{1}{2},
\frac{\sqrt{3}}{2}\big)\big)$ and the IFS $\{f_1,f_2\}$ given by
\[
f_1(x,y) = r R\Big(\frac{4\pi}{3}\Big) (x,y) + \Big(\frac{r}{2}, \frac{r\sqrt{3}}{2}\Big),\qquad f_2(x,y) = -r (x,y) + \Big(\frac{3r}{2}, \frac{r\sqrt{3}}{2}\Big),
\]
with $r \in (0, 1/\phi]$.
This family is non-overlapping, and outside for $r>1/2$.
Figure~\ref{figoutside} (where $r=1/\phi$) shows that it satisfies the conditions of Corollary~\ref{outsidecorollary}.
The dissection for the case $r=1/\phi$ is given in Figure~\ref{figoutside2}.
\end{example}

\begin{figure}[ht]
\includegraphics{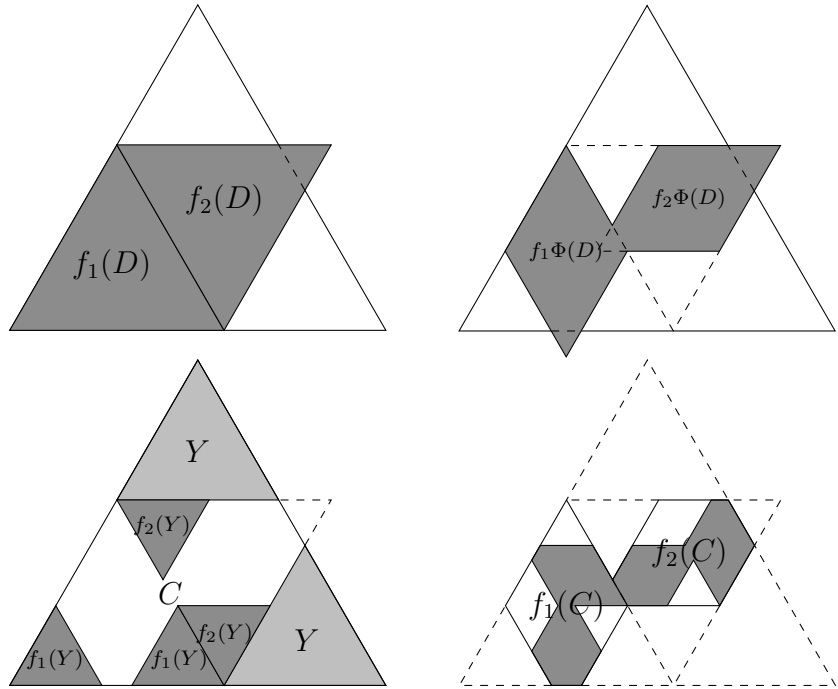}
\caption{The sets $\Phi(D), \Phi^2(D), Y, \Phi(Y), C, \Phi(C)$ for $r=1/\phi$ in Example~\ref{exaoutside}.}
\label{figoutside}
\end{figure}

\begin{figure}[ht]
\includegraphics{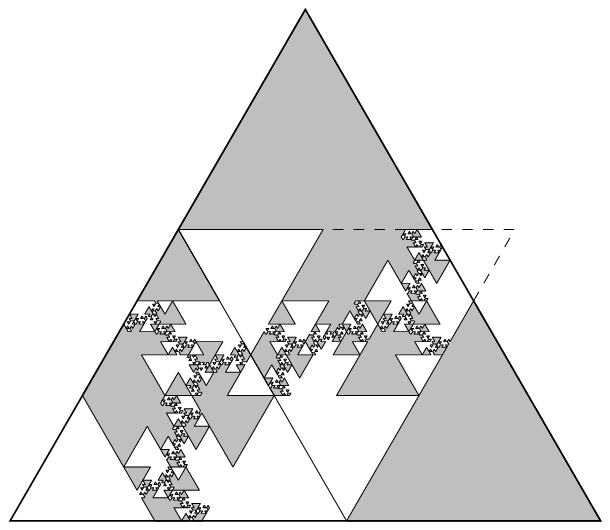}
\caption{The dissection for $r=1/\phi$ in Example~\ref{exaoutside}.}
\label{figoutside2}
\end{figure}
\end{section}

\begin{section}{Problems for Further Study} \label{problems}
We want to finish this paper with some questions and conjectures that are related to the topic of the present paper.
First of all, some part of the question we stated at the beginning remains unsolved.

\begin{question}
Can we dissect an equilateral triangle in three similar parts having
area ratio  $1:1:a$ for some $a \in
\big(1,\frac{3+\sqrt{5}}{2}\big)$?
\end{question}

\begin{question}
Can we dissect a square in three similar parts having area ratio
$1:1:a$ for some $a \in \big(1,\phi^2\big)$, where $\phi$ denotes
``high phi'', the positive root of $x^3-2x^2+x-1$?
\end{question}

We want to generalize these questions. To this matter, let $r(f)$
denote the contraction ratio of a contractive mapping $f$.

\begin{question}
Let $D \subset \mathbb{R}^d$, $d \ge 2$, with $D = \overline{D^\circ}$ and $k \ge 1$ be given.
Find the smallest constant $B<1$ depending on $D$ and $k$ with the following property.
There exist only finitely many families $\{f_1,\ldots,f_k\}$ of contractions satisfying $\min_i r(f_i) > B$ that give rise to a dissection of~$D$.
\end{question}

The assumption $d\ge 2$ cannot be dropped. For $D=[0,1]\subset
\mathbb{R}$, it is trivially possible to dissect $D$ into similar
$k$ intervals by any ratio $r_1:r_2:\dots:r_k$. To state a more
concrete variant of this question, we conjecture that, for each
$D\subset \mathbb{R}^d$ with $D = \overline{D^\circ}$, there are
only finitely many families $\{f_1,\ldots,f_k\}$ solving the
dissection problem for $D$ and satisfying $\sum_{i=1}^k r(f_i)^d >
1$. We call such solutions ``sporadic'' solutions of the dissection
problem. The solution depicted in Figure~\ref{rectangle} seems to be
such a sporadic solution.

\begin{question}
Let $D$ and $\mathcal{F}$ be given.
Can we find an algorithm for deciding whether there exists a dissection?
\end{question}
\end{section}

We thank Arturas Dubickas and Charlene Kalle for stimulating discussions.
Especially, Figure~\ref{approximation2} with $r=1/2$ is due to Dubickas.
We are deeply indebted to Tohru Tanaka at Meikun high school who informed us about the reference~\cite{Gardner:94}.

\bibliographystyle{siam}
\bibliography{landtiling}
\vspace{-7mm}
\end{document}